\documentclass[final,5p,times,twocolumn]{elsarticle}

\usepackage[T1]{fontenc}
\usepackage{amsmath,amssymb,mathtools,amsthm}
\usepackage{tikz-cd}
\usepackage{graphicx} 
\usepackage{float}    
\usepackage{microtype}
\microtypesetup{expansion=false}
\usepackage[hidelinks]{hyperref}
\usepackage{xcolor}
\usepackage{comment}
\journal{Systems \& Control Letters}
\numberwithin{equation}{section}
\newtheorem{theorem}{Theorem}
\newtheorem{proposition}[theorem]{Proposition}
\newtheorem{lemma}[theorem]{Lemma}
\theoremstyle{remark}
\newtheorem{remark}[theorem]{Remark}

\newcommand{\A}{{\cal C}}
\newcommand{\C}{\mathbb C}

\newcommand{\X}{\mathcal X}
\newcommand{\rank}{\operatorname{rank}}

\def \beq{\begin{equation}}
\def \eeq{\end{equation}}

\begin{document}

\begin{frontmatter}

\title{Boundary null controllability of weakly coupled structurally damped beams}

\author[aff1]{Julian Edward\corref{cor1}}\ead{edwardj@fiu.edu}\author[aff2,aff3]{Yacouba Simpore}\cortext[cor1]{Corresponding author.}\affiliation[aff1]{organization={Department of Mathematics and Statistics, Florida International University}, city={Miami}, state={FL}, postcode={33199}, country={USA}}

\affiliation[aff2]{organization={Universit\'e Yembila Abdoulaye Toguyeni},city={Fada N'Gourma}, country={Burkina Faso}}

\affiliation[aff3]{organization={Chair for Dynamics, Control, Machine Learning and Numerics, Friedrich-Alexander-Universit\"at Erlangen--N\"urnberg}, city={Erlangen}, postcode={91058}, country={Germany}}

\begin{abstract}
We prove boundary null controllability in any positive time for a finite system of Euler--Bernoulli beams with possibly non-diagonalizable coupling. For $0<\rho<2$, a Fourier--Jordan reduction yields a vector moment problem with polynomial--exponential modes. Under the Fattorini--Hautus condition, global spectral non-collision, and nonvanishing modal transfer
factors, a suitably estimated block-biorthogonal family produces an $H^2$ boundary control. Geometric multiplicities of the coupling matrix determine the minimal control dimension,  and Jordan-block lengths determine the degrees of the
 generalized moments. At $\rho=2$, an exact heat-system reduction gives the same result for $A=A^*\geq0$, despite the time-differential linkage
of the reduced boundary inputs.
\end{abstract}


\end{frontmatter}

\section{Introduction}
Motivated by its importance in engineering, control and stabilization problems for systems of Euler--Bernoulli beams have been the object of much study, see for instance \cite{LLS}, \cite{DZ}, \cite{GW}. However, it seems that very little is known about the behaviour of such systems in the presence of structural damping. 

Let $L=-\partial_{xx}$ be the Dirichlet Laplacian in $L^2(0,\pi)$ and
$Q=(0,\pi)\times(0,T)$. We consider
\begin{equation}
\label{eq:beam}
\left\{
\begin{aligned}
 u_{tt}+L^2u+\rho Lu_t+Au&=0 &&\text{in }Q,\\
 u(0,t)=u_{xx}(0,t)&=0,\\
 u(\pi,t)=Bf(t),\quad u_{xx}(\pi,t)&=0,\\
 u(\cdot,0)=u^0,\quad u_t(\cdot,0)&=u^1,
\end{aligned}
\right.
\end{equation}
where $A\in\C^{N\times N}$, $B\in\C^{N\times\ell}$ and
$f:(0,T)\to\C^\ell$. The left end is hinged and homogeneous; at the right
end, the displacement is prescribed in the directions selected by $B$, while
the bending moment remains zero.

Structural damping places the beam equation between hyperbolic and
parabolic dynamics. The displacement still obeys a second-order-in-time equation, but the damping acts at the same spatial order as the square root of the elastic operator and suppresses high frequencies at a Gaussian rate. This regime is natural for interconnected flexible structures controlled through a small number of boundary actuators. Our objective is to bring both displacement and velocity exactly to zero at a prescribed time.

Null controllability is governed by two mechanisms. Algebraically, every eigendirection of the coupling matrix must be detected by the control matrix, as expressed by the Fattorini--Hautus condition. Spectrally, separation of the spatial variables transforms the terminal conditions into infinitely many moments against the temporal roots of the damped beam. Their solvability requires a biorthogonal family whose growth is compensated by the Gaussian decay of the free solution.

For $0<\rho<2$, the roots form two complex parabolic branches with real part $-\rho m^2/2$. At $\rho=2$, the principal operator factorizes as $(\partial_t+L)^2$, and the beam equation can be reduced to a first-order parabolic system. Its boundary inputs remain linked as $(Bf,Bf')$, however, so controllability results with two independent parabolic controls do not directly apply.

Fourier--moment methods for Dirichlet control of scalar structurally damped beams were
developed in \cite{AE}. Stabilization for systems of strings or beams with indirect damping was studied in \cite{ACK}, \cite{A}, \cite{KLT}. Estimates for biorthogonal families and the effects of spectral condensation are studied in
\cite{ABGT}. Related controllability questions for
coupled parabolic systems are treated in
\cite{AKBGT11,Fernandez,Duprez}. The present problem combines these settings while allowing the coupling matrix to be non-diagonalizable.
A Jordan chain of length $r$ replaces a single exponential moment by
 $e^{\lambda t},\;te^{\lambda t},\ldots,t^{r-1}e^{\lambda t}.$
Thus geometric multiplicities determine the required number of control directions, whereas Jordan-block lengths determine the degrees of the generalized moments. Moreover, branches associated with distinct eigenvalues may be separated only by $O(m^{-2})$, which rules out a global uniform-gap argument.
We identify the minimal algebraic control dimension, establish the
Fourier--Jordan reduction and Gaussian terminal estimates, and solve the clustered moment problem by adapting the block method of
\cite{BoyerMorancey2023}. Finally, for $\rho=2$ and $A=A^*\ge0$, we show that the linkage between $f$ and $f'$ does not obstruct controllability.
Section~\ref{sec:problem} states the main results and the non-resonance conditions. Sections~\ref{sec:reduction}--\ref{sec:moments} treat $0<\rho<2$, and Section~\ref{sec:critical} treats $\rho=2$.

\section{Problem and main results}
\label{sec:problem}
The four temporal traces imposed below make
the lifting compatible with both the initial and terminal homogeneous states.

Let
$\varphi_m(x):=\sqrt{\frac{2}{\pi}}\sin(mx)$,
$L\varphi_m=m^2\varphi_m\hbox{, }m\ge1.$
For \(s\in\mathbb R\), define
\[
X^s
:=
\left\{
y=\sum_{m\ge1}y_m\varphi_m:
\sum_{m\ge1}m^{2s}|y_m|^2<\infty
\right\},
\]
\begin{equation}
\label{eq:Xs-norm}
\|y\|_{X^s}^2
:=
\sum_{m\ge1}m^{2s}|y_m|^2.
\end{equation}
and set \(\X^s:=(X^s)^N\). For \(s\ge0\), this agrees with
\(D(L^{s/2})\); for \(s<0\), it is understood as the corresponding
spectral completion. Finally,
\[
H_0^2(0,T;\C^\ell)
:=\{f\in H^2(0,T;\C^\ell):
f^{(q)}(0)=f^{(q)}(T)=0,\ q=0,1\}.
\]
This control space is dictated by the boundary lifting, not imposed merely for convenience. Indeed, substituting $u=\frac{x}{\pi}Bf+v$ creates the forcing term  $\frac{x}{\pi}Bf''$; hence $f''\in L^2$ is the natural requirement at the energy level used here. The traces at $t=0$ prevent the lifting from modifying the prescribed initial data. Those at $t=T$ ensure that nullity of the homogeneous
variable $v$ is equivalent to nullity of both $u$ and $u_t$. They will be enforced by only two additional scalar moment functions, independently of the number of beams and Jordan blocks. The lack of regularity of $u$ for $f\in L^2(0,T)$ has been observed in \cite{T1}, also \cite{AE}.

Let $\sigma(A)=\{\gamma_1,\ldots,\gamma_{p_0}\}$, not counting multiplicities. Denote by $g_j$ the
geometric multiplicity of $\gamma_j$ and by
$r_{j,1},\ldots,r_{j,g_j}$ its Jordan-block sizes; set
\begin{equation}
\label{eq:Jordan-data}
 g=\max_jg_j,
 \quad \nu_j=\max_{1\le k\le g_j}r_{j,k}.
\end{equation}
Thus the algebraic multiplicity of $\gamma_j$ is
$\sum_{k=1}^{g_j}r_{j,k}$. 
The Fattorini-Hautus (or equivalently  Kalman) condition is
\begin{equation}
\label{eq:Hautus}
 \rank[\lambda I_N-A\mid B]=N,
 \quad \lambda\in\sigma(A).
\end{equation}
Equivalently,
$\ker(\overline\lambda I_N-A^*)\cap\ker B^*=\{0\}$. In particular,
\eqref{eq:Hautus} implies $\ell\ge g$; equality corresponds to the minimal
number of control components.

\begin{proposition}
\label{prop:necessity}
If \eqref{eq:beam} is null controllable for every initial state, then
\eqref{eq:Hautus} holds. Consequently, at least $g$ scalar boundary controls
are necessary. Moreover, this lower bound is sharp.
\end{proposition}

\begin{proof}
If \eqref{eq:Hautus} fails, there are an eigenvalue $\gamma$ and a nonzero
$\psi\in\ker(\overline\gamma I_N-A^*)\cap\ker B^*$. The scalar projection
$z=\langle u,\psi\rangle$ then satisfies
\[
 z_{tt}+L^2z+\rho Lz_t+\gamma z=0
\]
with homogeneous boundary conditions: the controlled trace vanishes because
$\langle Bf,\psi\rangle=\langle f,B^*\psi\rangle=0$. Choose nonzero initial
data supported on one spatial eigenfunction in the direction $\psi$. The
corresponding Fourier coefficient solves an autonomous second-order ODE. If
its position and velocity vanished at $T$, uniqueness backward from $T$ would
force the initial coefficient to vanish, a contradiction. Thus every
eigendirection must be detected by $B^*$. Since this map is injective from an
eigenspace of dimension $g_j$ into $\C^\ell$, one has $\ell\ge g_j$ for all
$j$. Conversely, let
$E_j:=\ker(\overline{\gamma_j}I_N-A^*)$
and choose a basis
$\{\psi_{j,k}\}_{k=1}^{g_j}$ of $E_j$. Let
$\{e_1,\ldots,e_g\}$ be the canonical basis of $\mathbb C^g$.
Since the eigenspaces $E_j$ are in direct sum, define on
$\bigoplus_jE_j$
\[
 B_g^*\psi_{j,k}:=e_k,
 \quad 1\leq k\leq g_j,
\]
and extend $B_g^*$ linearly to $\mathbb C^N$. Then
\[
 E_j\cap\ker B_g^*=\{0\}
 \quad\text{for every }j.
\]
Hence $B_g\in\mathbb C^{N\times g}$ satisfies~(2.4), and $g$ is
the smallest number of columns for which the
Fattorini--Hautus condition can hold.
\end{proof}

In the next section, will show that  the temporal roots are
\begin{equation}
\label{eq:roots}
 \lambda_{j,m}^{\pm}
 =\frac{-\rho m^2\pm
 \sqrt{\rho^2m^4-4(m^4+\gamma_j)}}{2}, \ m\geq 1.
\end{equation}
If $0<\rho<2$, then, uniformly in $j$,
\begin{equation}
 \lambda_{j,m}^{\pm}
 =\frac{-\rho\pm i\sqrt{4-\rho^2}}2m^2+O(m^{-2}).
\end{equation}
In particular, for every $c<\rho/2$ and all sufficiently large $m$,
\begin{equation}
\label{eq:root-decay}
 \Re\lambda_{j,m}^{\pm}\le-cm^2,
 \quad
 |\lambda_{j,m}^+-\lambda_{j,m}^-|\asymp m^2.
\end{equation}
At $\rho =2$, the temporal roots simplify to
 $\lambda_{j,m}^{\pm}=-m^2\pm\sqrt{-\gamma_j},$
which anticipates the parabolic reduction of Section~\ref{sec:critical}.

The subcritical construction requires two spectral exclusions.
First, roots attached to distinct pairs must not coincide:
\begingroup
\renewcommand{\theHequation}{NR1}
\begin{equation}
 \tag{\rm NR$_1$}\label{nr1}
\lambda_{j,m}^{\sigma}\neq\lambda_{k,n}^{\eta}
 \quad\text{if }(j,m)\neq(k,n).
\end{equation}
\endgroup
Here $1\le j,k\le p_0$, $m,n\ge1$, and
$\sigma,\eta\in\{+,-\}$. The indices $j,k$ label the distinct
eigenvalues of $A$; their multiplicities are encoded by the Jordan
data. Condition {\rm (NR$_1$)} allows
$\lambda_{j,m}^{+}=\lambda_{j,m}^{-}$.
Set
\[
 \begin{aligned}
 \mathcal D_\rho
 &:=\{(j,m):\lambda_{j,m}^{+}=\lambda_{j,m}^{-}\}=\left\{(j,m):
 \gamma_j=-\frac{4-\rho^2}{4}m^4\right\}.
 \end{aligned}
\]
Since $0<\rho<2$, $\mathcal D_\rho$ is finite and
$\#\mathcal D_\rho\le p_0$. For $(j,m)\in\mathcal D_\rho$, put
 $\lambda_{j,m}^{\mathrm c}
 :=\lambda_{j,m}^{+}
 =\lambda_{j,m}^{-}
 =-\frac{\rho m^2}{2}.$
Define
\[
 \Sigma_{j,m}:=
 \begin{cases}
 \{+,-\},&(j,m)\notin\mathcal D_\rho,\\
 \{\mathrm c\},&(j,m)\in\mathcal D_\rho,
 \end{cases}
 \quad
 d_{j,m}^{\sigma}:=
 \begin{cases}
 \nu_j,&\sigma=\pm,\\
 2\nu_j,&\sigma=\mathrm c.
 \end{cases}
\]
Thus, if $\lambda_{j,m}^{+}=\lambda_{k,n}^{-}$, then
{\rm (NR$_1$)} gives $(j,m)=(k,n)$, and the two roots are represented
by the single label $\mathrm c$.

Second, we assume
\begingroup
\renewcommand{\theHequation}{NR2}
\begin{equation}
 \tag{\rm NR$_2$}\label{nr2}
 (\lambda_{j,m}^{\sigma})^2+\gamma_j\neq0,
 \quad
 \substack{1\le j\le p_0,\ m\ge1,\\
           \sigma\in\Sigma_{j,m}.}
\end{equation}
\endgroup
Set $\mathcal R_\rho:=\{-m^4/\rho^2:m\ge1\}$. 
This condition will guarantee that the following equation is non-trivial, for $f\in H_0^2(0,T;\C^\ell)$,
\[
 \begin{aligned}
 &\int_0^T(f''+\gamma_jf)e^{\lambda(T-t)}dt=
 (\lambda^2+\gamma_j)
 \int_0^Tf e^{\lambda(T-t)}dt,
 \end{aligned}
\]
where $\lambda=\lambda_{j,m}^{\sigma}$.
The characteristic
equation gives
\[
 \eqref{nr2}
 \quad\Longleftrightarrow\quad
 \sigma(A)\cap\mathcal R_\rho=\varnothing.
\]
Put $\tau=T-t$ and choose $\alpha\in\C$ such that
\[
 (\lambda_{j,m}^{\sigma})^2+\alpha\neq0,
 \quad
 \substack{1\le j\le p_0,\ m\ge1,\\
\sigma\in\Sigma_{j,m}.}
\]
Let $\omega^2=-\alpha$. For
$\sigma\in\Sigma_{j,m}$ and $0\le q<d_{j,m}^{\sigma}$, define
\begin{equation}
 E_{j,q,m}^{\sigma}(t)
 :=\tau^q e^{\lambda_{j,m}^{\sigma}\tau}.
 \label{fam}
\end{equation}
Hence a double root gives one confluent family of order $2\nu_j$.
Finally, define
\[
 (\chi_\alpha^0,\chi_\alpha^1):=
 \begin{cases}
 (e^{\omega\tau},e^{-\omega\tau}),&\omega\neq0,\\
 (1,\tau),&\omega=0.
 \end{cases}
\]
These auxiliary modes encode the terminal traces of
$f''+\alpha f=G$.

\begin{lemma}
\label{lem:biorthogonality}
Let $T>0$ and $0<\rho<2$. Under \eqref{nr1}, and with $\alpha$
fixed as above, there exist functions
$g_{j,q,m}^{\sigma}\in L^2(0,T)$ satisfying
\begin{align}
 \int_0^T E_{p,r,n}^{\eta}g_{j,q,m}^{\sigma}\,dt
 &=\delta_{jp}\delta_{qr}\delta_{mn}\delta_{\sigma\eta},
 \label{eq:block2}\\
 \int_0^T\chi_\alpha^d g_{j,q,m}^{\sigma}\,dt
 &=0,\qquad d=0,1,
 \label{eq:block3}\\
 \|g_{j,q,m}^{\sigma}\|_{L^2(0,T)}
 &\le C_T(1+m)^M e^{C_Tm}
 \label{eq:block4}
\end{align}
for all admissible indices, where $M\ge0$ and $C_T>0$ are
independent of $j,q,m,\sigma$. Thus the functions
$g_{j,q,m}^{\sigma}$ are biorthogonal to the family in
\eqref{fam} and orthogonal to $\chi_\alpha^0,\chi_\alpha^1$.
The temporal pairings are bilinear, without complex conjugation.
\end{lemma}

\begin{proof}
Since $\mathcal D_\rho$ is finite, all sufficiently large modes have
two distinct roots. Choosing the labels $+$ and $-$ consistently
with the two high-frequency branches and setting
$d_\rho=(4-\rho^2)^{1/2}$, we have, uniformly in $j$,
\begin{equation}
 -\lambda_{j,m}^{\pm}
 =
 \frac{\rho\mp i d_\rho}{2}m^2
 \mp\frac{i\gamma_j}{d_\rho}m^{-2}
 +O(m^{-6}).
 \label{asy}
\end{equation}
Since the $\gamma_j$ are pairwise distinct, \eqref{asy} yields, for
all sufficiently large indices,
\begin{align}
 |\lambda_{j,m}^{\sigma}-\lambda_{k,m}^{\sigma}|
 &\asymp m^{-2},
 &&j\neq k,                                      \label{gap}\\
 |\lambda_{j,m}^{\sigma}-\lambda_{k,n}^{\sigma}|
 &\ge c|m^2-n^2|,
 &&m\neq n,                                      \notag\\
 |\lambda_{j,m}^{+}-\lambda_{k,n}^{-}|
 &\ge c(m^2+n^2).                                \notag
\end{align}
Thus high-frequency condensation occurs only within groups having fixed $(m,\sigma)$.
Choose $a_0>0$ sufficiently large, set
 $\mu_{j,m}^{\sigma}:=a_0-\lambda_{j,m}^{\sigma}\hbox{, }\sigma\in\Sigma_{j,m},$
and consider
\[
 \Lambda_*
 :=
 \{\mu_{j,m}^{\sigma}:
 1\le j\le p_0,\ m\ge1,\
 \sigma\in\Sigma_{j,m}\}\cup A_*,
 \]
 where \(A_*=\{a_0\pm\omega\}.\)
By \eqref{nr1}, the convention that a double root is counted only
once, and the choice of $\alpha$, the underlying nodes of
$\Lambda_*$ are pairwise distinct. Moreover, since
 $\left|\frac{\rho\mp i d_\rho}{2}\right|=1,$
\eqref{asy} gives, uniformly in $j$ and $\sigma$,
\[
 \Re\mu_{j,m}^{\sigma}\asymp m^2,
 \qquad
 |\mu_{j,m}^{\sigma}|
 =
 m^2+\frac{\rho a_0}{2}+O(m^{-2}).
\]
Together with \eqref{gap}, these estimates verify the sector,
local weak-gap, and counting conditions in
\cite[(22), (23), (28)--(30)]{BoyerMorancey2023}. Hence, for
suitable $p\in\mathbb N^*$ and $\varrho,\tau_0,\kappa>0$,
\[
 \Lambda_*
 \in\mathcal L_w(p,\varrho,\tau_0,1/2,\kappa).
\]
The grouping of \cite[Proposition~6]{BoyerMorancey2023} can be
chosen, outside finitely many groups, as
\[
 G_{m,\sigma}
 =
 \{\mu_{k,m}^{\sigma}:1\le k\le p_0\},
 \quad \sigma\in\{+,-\}.
\]
Assign multiplicity $d_{j,m}^{\sigma}$ to
$\mu_{j,m}^{\sigma}$, multiplicity $1$ to each node
$a_0\mp\omega$ if $\omega\neq0$, and multiplicity $2$ to $a_0$
if $\omega=0$. Set
 $d_*:=\max\{2,2\nu_1,\ldots,2\nu_{p_0}\}.$
Fix $(j,q,m,\sigma)$ with $\sigma\in\Sigma_{j,m}$ and
$0\le q<d_{j,m}^{\sigma}$. Apply
\cite[Theorem~46]{BoyerMorancey2023} to the group containing
$\mu_{j,m}^{\sigma}$, taking the parameter of that theorem equal
to $d_*$ and using the assigned multiplicities as multi-index.
Prescribe the datum $1$ at $\mu_{j,m}^{\sigma}$ and order $q$,
and zero at every other node and order.

The interpolation and vanishing assertions {\rm (115a)--(115b)}
give $h_{j,q,m}^{\sigma}\in L^2(0,T)$ such that
\begin{equation}
 \int_0^T
 \frac{(-s)^r}{r!}
 e^{-\mu_{k,n}^{\vartheta}s}
 h_{j,q,m}^{\sigma}(s)\,ds
 =
 \delta_{jk}\delta_{qr}\delta_{mn}
 \delta_{\sigma\vartheta}
 \label{mom}
\end{equation}
for every admissible $(k,r,n,\vartheta)$, together with zero
moments of order $0$ at $a_0\mp\omega$ if $\omega\neq0$, and of
orders $0,1$ at $a_0$ if $\omega=0$.

Let
\[
 R:=\sum_{k=1}^{p_0}\nu_k,
 \quad
 M:=2(R-1).
\]
For a non-exceptional group, the multiplicity multi-index is
$(\nu_1,\ldots,\nu_{p_0})$ and its total multiplicity is $R$.
Standard confluent divided-difference estimates, together with \eqref{gap}, bound the divided-difference factor in
\cite[(116)]{BoyerMorancey2023} by
 $C(1+m)^M.$ Moreover, if $r_G:=\inf_{\mu\in G}\Re\mu$, then
 $r_{G_{m,\sigma}}\asymp m^2.$ Since the counting exponent is $1/2$, estimate {\rm (116)} yields
\begin{equation}
 \|h_{j,q,m}^{\sigma}\|_{L^2(0,T)}
 \le C_T(1+m)^M e^{C_Tm}.
 \label{hn}
\end{equation}
The finitely many groups containing a double, auxiliary, or
low-frequency node are absorbed into $C_T$.
Finally, define
 $g_{j,q,m}^{\sigma}(t)
 :=
 \frac{(-1)^q}{q!}
 e^{-a_0(T-t)}
 h_{j,q,m}^{\sigma}(T-t).$
With $s=T-t$, \eqref{mom} gives
\begin{align*}
 \int_0^T
 E_{k,r,n}^{\vartheta}(t)
 g_{j,q,m}^{\sigma}(t)\,dt
 &=
 \frac{(-1)^{q+r}r!}{q!}
 \delta_{jk}\delta_{qr}\delta_{mn}
 \delta_{\sigma\vartheta}=
 \delta_{jk}\delta_{qr}\delta_{mn}
 \delta_{\sigma\vartheta},
\end{align*}
which proves \eqref{eq:block2}. The auxiliary zero moments give \eqref{eq:block3}. Finally, time reversal is an isometry and
 $\frac{|e^{-a_0(T-t)}|}{q!}\le1.$
Hence \eqref{hn} implies \eqref{eq:block4}.
\end{proof}
\begin{theorem}
\label{thm:subcritical}
Let $T>0$ and $0<\rho<2$. Assume the Fattorini--Hautus condition
\eqref{eq:Hautus}, the global spectral non-resonance condition
\eqref{nr1}, and the transfer non-resonance
condition \eqref{nr2}. Then every
$(u^0,u^1)\in\X^3\times\X^1$
can be driven to rest by a control
 $f\in H_0^2(0,T;\C^\ell).$
Moreover,
\begin{equation}
\label{eq:control}
 \|f\|_{H^2(0,T;\C^\ell)}
 \le
 C_T
 \left(
 \|u^0\|_{\X^3}
 +
 \|u^1\|_{\X^1}
 \right).
\end{equation}
\end{theorem}

\begin{remark}
If $\sigma(A)\subset[0,\infty)$, conditions
\eqref{nr1}--\eqref{nr2} are
automatic. Indeed, the temporal roots have real part $-\rho m^2/2$, so a collision
first forces the same spatial index; their imaginary parts then distinguish
the eigenvalue and the sign. No diagonalizability of $A$ is needed.
\end{remark}

\begin{theorem}
\label{thm:critical}
Let $T>0$, $\rho=2$, and $A=A^*\ge0$. If \eqref{eq:Hautus} holds, then every
$(u^0,u^1)\in\X^3\times\X^1$ can be driven to rest by some
$f\in H_0^2(0,T;\C^\ell)$ satisfying \eqref{eq:control}.
\end{theorem}

\section{Fourier--Jordan reduction}
\label{sec:reduction}

Choose adjoint Jordan chains
\begin{equation}
\label{eq:adjoint-chain}
 A^*\psi_{j,k,s}
 =\overline{\gamma_j}\psi_{j,k,s}+\psi_{j,k,s-1},
 \quad \psi_{j,k,0}=0,
\end{equation}
and let $\{\theta_{j,k,s}\}$ be the biorthogonal basis, normalized by
\[
\langle\theta_{j,k,s},\psi_{j',k',s'}\rangle_{\C^N}
 =\delta_{jj'}\delta_{kk'}\delta_{ss'}.
\]
We use the Hermitian product linear in the first argument. With this
convention,
$A\theta_{j,k,s}=\gamma_j\theta_{j,k,s}+\theta_{j,k,s+1}\hbox{, }
\theta_{j,k,r_{j,k}+1}=0.$
Setting $\beta_{j,k,s}=B^*\psi_{j,k,s}$, condition \eqref{eq:Hautus} gives
\begin{equation}
\label{eq:Bhat-rank}
\rank
\begin{bmatrix}
\beta_{j,1,1} & \cdots & \beta_{j,g_j,1}
\end{bmatrix}
= g_j.
\end{equation}
Indeed, the vectors $\{\psi_{j,k,1}\}_{k=1}^{g_j}$ form a basis of $\ker(\overline{\gamma_j}I_N-A^*)$, and \eqref{eq:Hautus} implies that $B^*$ is injective on this eigenspace. Hence the columns of the matrix in \eqref{eq:Bhat-rank} are linearly
independent. In particular, $g_j\le\ell$ for every $j$.
Let $f\in H^2(0,T;\C^\ell)$ satisfy $f(0)=f'(0)=0$.
Set $h(x)=x/\pi$ and write
 $u=hBf+v.$
Then $v$ satisfies homogeneous hinged boundary conditions and
solves
\begin{equation}
\label{eq:lifted-system}
 \begin{aligned}
 v_{tt}+L^2v+\rho Lv_t+Av&=-h(Bf''+ABf),\\
 (v(0),v_t(0))&=(u^0,u^1).
 \end{aligned}
\end{equation}
Here the lifting is understood at the level of the differential expression:
$h''=h^{(4)}=0$ in $(0,\pi)$. We do not use the operator identity $Lh=0$,
because $h$ does not belong to the domain of the homogeneous Dirichlet
realization of $L$.
Split $v=v^a+v^b$ into the control-driven part with zero initial data and the
free part. Let $\varphi_m=\sqrt{2/\pi}\sin(mx)$ and
$\Phi_{j,k,s,m}=\varphi_m\theta_{j,k,s}$,
$\Psi_{j,k,s,m}=\varphi_m\psi_{j,k,s}$. These two families are biorthogonal
in $(L^2(0,\pi))^N$ and, since the Jordan bases are finite dimensional, their
coefficient norms are equivalent to the standard Fourier norms. Thus
\begin{equation}
 v^a=\sum_{j,k,s}\sum_{m\ge1}a_{j,k,s,m}\Phi_{j,k,s,m},
 \quad
 v^b=\sum_{j,k,s}\sum_{m\ge1}b_{j,k,s,m}\Phi_{j,k,s,m}.
\end{equation}
The sums over $k,s$ mean $1\le k\le g_j$ and
$1\le s\le r_{j,k}$. The coefficients are obtained by pairing with
$\Psi_{j,k,s,m}$. Since
\begin{equation}
\label{eq:hm}
 h=\sum_{m\ge1}x_m\varphi_m,
 \quad x_m=(-1)^{m+1}\sqrt{2/\pi}\,m^{-1},
\end{equation}
the coefficients satisfy, with $F_j=f''+\gamma_jf$,
\begin{equation}
\label{eq:modal-system}
\left\{
\begin{aligned}
 \mathcal P_{j,m}a_{j,k,s,m}
 &=-a_{j,k,s-1,m}
 -x_m(\beta_{j,k,s}^*F_j+\beta_{j,k,s-1}^*f),\\
 \mathcal P_{j,m}b_{j,k,s,m}&=-b_{j,k,s-1,m},
\end{aligned}
\right.
\end{equation}
where
$\mathcal P_{j,m}=\partial_t^2+\rho m^2\partial_t+m^4+\gamma_j$ and all
level-zero coefficients vanish. The $a$-coefficients have zero initial data,
whereas
\[
 b_{j,k,s,m}(0)=u^0_{j,k,s,m},
 \quad b'_{j,k,s,m}(0)=u^1_{j,k,s,m}.
\]
Here $u^r_{j,k,s,m}=\langle u^r,\Psi_{j,k,s,m}\rangle$, $r=0,1$. The term
$-a_{j,k,s-1,m}$ in \eqref{eq:modal-system} is precisely the nilpotent part of
the Jordan block; it is the origin of the triangular recursion in $s$.

For $(j,m)\notin\mathcal D_\rho$, set
$\A_{j,m}
:=
\bigl(\lambda_{j,m}^+-\lambda_{j,m}^-\bigr)^{-1}.$
For every pair $(j,m)$, define
\begin{align}
K_{j,m}(t)
&:=
\begin{cases}
\displaystyle
\A_{j,m}
\left(
e^{\lambda_{j,m}^+t}
-
e^{\lambda_{j,m}^-t}
\right),
&
(j,m)\notin\mathcal D_\rho,
\\[3mm]
\displaystyle
t e^{\lambda_{j,m}^{\mathrm c}t},
&
(j,m)\in\mathcal D_\rho,
\end{cases}
\label{eq:K}
\\
H_{j,m}(t)
&:=
\begin{cases}
\displaystyle
\A_{j,m}
\left(
\lambda_{j,m}^+
e^{\lambda_{j,m}^-t}
-
\lambda_{j,m}^-
e^{\lambda_{j,m}^+t}
\right),
&
(j,m)\notin\mathcal D_\rho,
\\[3mm]
\displaystyle
\left(
1-\lambda_{j,m}^{\mathrm c}t
\right)
e^{\lambda_{j,m}^{\mathrm c}t},
&
(j,m)\in\mathcal D_\rho.
\end{cases}
\notag
\end{align}
Then
\begin{equation}
\label{eq:kernel-initial-data}
\begin{aligned}
K_{j,m}(0)&=0,
&
K'_{j,m}(0)&=1,
\\
H_{j,m}(0)&=1,
&
H'_{j,m}(0)&=0.
\end{aligned}
\end{equation}
For every $T>0$ and $0<c<\rho/2$, after changing the constant to include
the finitely many low and double-root modes, there exists
$C_{T,c}>0$ such that, for $0\le t\le T$,
\begin{equation}
\label{eq:kernel-bounds}
 \begin{aligned}
 |K_{j,m}(t)|&\le C_{T,c}m^{-2}e^{-cm^2t},
 &|K'_{j,m}(t)|&\le C_{T,c}e^{-cm^2t},\\
 |H_{j,m}(t)|&\le C_{T,c}e^{-cm^2t},
 &|H'_{j,m}(t)|&\le C_{T,c}m^2e^{-cm^2t}.
 \end{aligned}
\end{equation}
If $*$ denotes causal convolution and
$R_{j,k,s,m}=H_{j,m}u^0_{j,k,s,m}+K_{j,m}u^1_{j,k,s,m}$, Duhamel's formula
reduces \eqref{eq:modal-system} to
\begin{align}
 b_{j,k,s,m}&=R_{j,k,s,m}-K_{j,m}*b_{j,k,s-1,m},
 \label{eq:b-recursion}\\
 a_{j,k,s,m}&=-K_{j,m}*\bigl[a_{j,k,s-1,m}\\[-1mm]
 &\hspace{13mm}
 +x_m(\beta_{j,k,s}^*F_j+\beta_{j,k,s-1}^*f)\bigr].
 \label{eq:a-recursion}
\end{align}
For later use, set
\[
 Q_{j,k,r,m}:=\beta_{j,k,r}^*F_j+\beta_{j,k,r-1}^*f,
 \quad Q_{j,k,0,m}:=0.
\]
An induction on $s$ gives the exact identity
\begin{equation}
\label{eq:a-iterated}
 a_{j,k,s,m}=x_m\sum_{r=1}^{s}(-1)^r
 K_{j,m}^{*r}*Q_{j,k,s-r+1,m};
\end{equation}
here $K_{j,m}^{*r}$ denotes $r$ iterations of convolution by $K_{j,m}$. Also,
\begin{equation}
\label{eq:b-iterated}
 b_{j,k,s,m}
 =\sum_{r=0}^{s-1}(-1)^r
 K_{j,m}^{*r}*R_{j,k,s-r,m},
\end{equation}
where $K^{*0}*R=R$. Iteration is finite, and the identities
\eqref{eq:a-iterated}--\eqref{eq:b-iterated}
are valid in both the simple- and double-root cases.

\begin{lemma}
\label{Kconv}
Let $s\ge1$, and let $K_{j,m}^{*s}$ denote the $s$-fold causal
convolution of $K_{j,m}$ with itself.
If $(j,m)\notin\mathcal D_\rho$, then
\begin{equation}
\label{Kfor}
K_{j,m}^{*s}(t)
=
\sum_{r=1}^{s}
d_{r,s}\A_{j,m}^{\,2s-r}
\frac{t^{r-1}}{(r-1)!}
\left[
(-1)^{s-r}e^{\lambda_{j,m}^{+}t}
+
(-1)^s e^{\lambda_{j,m}^{-}t}
\right],
\end{equation}
where $d_{r,s}:=\binom{2s-r-1}{s-r}.$
If $(j,m)\in\mathcal D_\rho$, then
\begin{equation}
\label{eq:double-convolution}
K_{j,m}^{*s}(t)
=
\frac{t^{2s-1}}{(2s-1)!}
e^{\lambda_{j,m}^{\mathrm c}t}.
\end{equation}
\end{lemma}
Proof: Assume first that $(j,m)\notin\mathcal D_\rho$. Applying the Laplace transform to $K_{j,m}^{*r}$, we get 
$\mathcal L(K_{j,m}^{*r})(z)
=
\frac{1}{
(z-\lambda_{j,m}^+)^r
(z-\lambda_{j,m}^-)^r
}.$
An exercise in partial fractions now yields the formula above in this case. 
Part B can be proven similarly, or by integration. The details are left to the reader. $\Box$

In the sequel, the two terms in \eqref{Kfor} corresponding to $r=s$, equivalently those containing $t^{s-1}$, will be called the \emph{principal terms}.
\begin{proposition}
\label{prop:wellposed}
Let $0<\rho < 2.$ 
For
$f\in H^2(0,T;\C^\ell)$ with $f(0)=f'(0)=0$ and
$(u^0,u^1)\in\X^3\times\X^1$, \eqref{eq:lifted-system} has a unique solution
such that
\begin{equation}
\label{eq:regularity}
 \begin{aligned}
 v\in{}&C([0,T];\X^3)\cap C^1([0,T];\X^1)\cap H^2(0,T;\X^{-1}),
 \end{aligned}
\end{equation}
\begin{equation}
\label{eq:wellposed}
 \begin{aligned}
 &\hbox{and }\|v\|_{C([0,T];\X^3)}+\|v_t\|_{C([0,T];\X^1)}+\|v_{tt}\|_{L^2(0,T;\X^{-1})}\\
 &\le C_T(\|u^0\|_{\X^3}+\|u^1\|_{\X^1}
 +\|f\|_{H^2(0,T;\C^\ell)}).
 \end{aligned}
\end{equation}
Moreover, for every $0<\eta<\rho$ and each fixed $j,k,s$,
\begin{equation}
\label{eq:Gaussian}
 \sum_{m\ge1}e^{\eta m^2T}
 \bigl(m^6|b_{j,k,s,m}(T)|^2+m^2|b'_{j,k,s,m}(T)|^2\bigr)
 \le C_{\eta,T}\mathcal E_0^2,
\end{equation}
where $\mathcal E_0=\|u^0\|_{\X^3}+\|u^1\|_{\X^1}$.
\end{proposition}

\begin{proof}
The estimates for the first level $s=1$ follow directly from
$b_{j,k,1,m}=R_{j,k,1,m}$ and \eqref{eq:kernel-bounds}; more precisely,
\begin{equation}
 \label{eq:R-bounds}
 \begin{aligned}
 |R_{j,k,s,m}(t)|&\le C_{T,c}e^{-cm^2t}
 (|u^0_{j,k,s,m}|+m^{-2}|u^1_{j,k,s,m}|),\\
 |R'_{j,k,s,m}(t)|&\le C_{T,c}e^{-cm^2t}
 (m^2|u^0_{j,k,s,m}|+|u^1_{j,k,s,m}|).
 \end{aligned}
\end{equation}
Furthermore, for $r\ge1$,
\begin{equation}
 \label{eq:convol1}
 |K_{j,m}^{*r}(t)|\le
 \frac{C_{T,c}^r}{(r-1)!}\,m^{-2r}t^{r-1}e^{-cm^2t}.
\end{equation}
This follows directly by iterating
\eqref{eq:kernel-bounds}, and applies to both simple and double roots.
Moreover, for $r\ge1$,
\begin{equation}
 |(K_{j,m}^{*r})'(t)|
 \le \frac{C_{T,c}^r}{(r-1)!}\,
 m^{-2(r-1)}t^{r-1}e^{-cm^2t}.
\end{equation}
For $r\ge2$ this follows from
$(K_{j,m}^{*r})'=K'_{j,m}*K_{j,m}^{*(r-1)}$; the case $r=1$ is \eqref{eq:kernel-bounds}. Inserting these estimates into \eqref{eq:b-iterated} and using the bounded Jordan-chain length, we obtain, for $0\le t\le T$,
\[
 m^3|b_{j,k,s,m}(t)|+m|b'_{j,k,s,m}(t)|
 \le C_{T,c}e^{-cm^2t}
 \sum_{q=1}^{s}
 \left(
 m^3|u^0_{j,k,q,m}|+m|u^1_{j,k,q,m}|
 \right).
\]
Choose $c$ such that $\eta/2<c<\rho/2$. Evaluating at $t=T$,
squaring, multiplying by $e^{\eta m^2T}$, and summing over $m$ give \eqref{eq:Gaussian}, since $2c-\eta>0$.
For the controlled part, $v^a$, \eqref{eq:a-recursion} has zero initial data and a
forcing coefficient $x_m=O(m^{-1})$. From \eqref{eq:kernel-bounds}, with
constants allowed to depend on $T$,
\[
 \|K_{j,m}\|_{L^2(0,T)}\le Cm^{-3},
 \quad \|K'_{j,m}\|_{L^2(0,T)}\le Cm^{-1}.
\]
We also have
\[
 \|K_{j,m}\|_{L^1(0,T)}\le C_Tm^{-4},
 \quad \|K'_{j,m}\|_{L^1(0,T)}\le C_Tm^{-2}.
\]
Cauchy--Schwarz in the convolution therefore gives, at the first chain level,
$\sup_t|a_{j,k,1,m}(t)|\le C_Tm^{-4}\|F_j\|_{L^2}$ and
$\sup_t|a'_{j,k,1,m}(t)|\le C_Tm^{-2}\|F_j\|_{L^2}$. We now argue by
induction on the chain level. If the same estimates hold at level $s-1$, the
term $K_{j,m}*a_{j,k,s-1,m}$ gains an additional convolution, whereas the
direct forcing
$x_mK_{j,m}*(\beta_{j,k,s}^*F_j+\beta_{j,k,s-1}^*f)$ satisfies the same
bounds as at level one. Hence, for every admissible $s$,
\begin{equation}
\label{eq:controlled}
 \sup_{0\le t\le T}|a_{j,k,s,m}(t)|
 \le C_Tm^{-4}\|f\|_{H^2(0,T)}\hbox{,
} \sup_{0\le t\le T}|a'_{j,k,s,m}(t)|
 \le C_Tm^{-2}\|f\|_{H^2(0,T)}.
\end{equation}
The series weighted by $m^6$ and $m^2$ are therefore summable. Moreover, the
modal equation gives
\[
 a''_{j,k,s,m}=-\rho m^2a'_{j,k,s,m}-(m^4+\gamma_j)a_{j,k,s,m}
 -a_{j,k,s-1,m}
 -x_m(\beta_{j,k,s}^*F_j+\beta_{j,k,s-1}^*f).
\]
Using \eqref{eq:controlled} and summing with the
$\mathcal X^{-1}$ weight $m^{-2}$ yields
\[
 \sum_{m\ge1}m^{-2}\|a''_{j,k,s,m}\|_{L^2(0,T)}^2
 \le C_T\|f\|_{H^2(0,T)}^2.
\]
The preceding bounds and the modal equation give
\[
 \|b''_{j,k,s,m}\|_{L^2(0,T)}
 \le C_T\sum_{r=1}^{s}
 \left(m^3|u^0_{j,k,r,m}|+m|u^1_{j,k,r,m}|\right).
\]
Consequently,
$\sum_{m\ge1}m^{-2}\|b''_{j,k,s,m}\|_{L^2(0,T)}^2
 \le C_T\mathcal E_0^2.$
Parseval's identity, equivalence of the finite-dimensional Jordan norms, and dominated convergence then prove
\eqref{eq:regularity}--\eqref{eq:wellposed}. Uniqueness follows mode by mode from \eqref{eq:modal-system}.
\end{proof}

\section{Generalized moments}
\label{sec:moments}
We retain the notation
$\mathcal D_\rho$, $\lambda_{j,m}^{\mathrm c}$,
$\Sigma_{j,m}$, $d_{j,m}^{\sigma}$, $\nu_j$, $\tau$,
and $E_{j,q,m}^{\sigma}$ introduced in Section~2.
All temporal moment pairings are bilinear; no complex conjugation is
applied to $E_{j,q,m}^{\sigma}$.

Because $f(T)=f'(T)=0$, null controllability is equivalent to
\begin{equation}
\label{eq:terminal}
\begin{aligned}
a_{j,k,s,m}(T)&=-b_{j,k,s,m}(T),\\
a'_{j,k,s,m}(T)&=-b'_{j,k,s,m}(T).
\end{aligned}
\end{equation}
We now reduce the null-controllability problem to the following moment problem.
\begin{proposition}
\label{prop:moment}
Let $0<\rho< 2$. Assume \eqref{eq:Hautus} and
\eqref{nr1}--\eqref{nr2}. There exist vectors
\[
\zeta_{j,q,m}^{\sigma}\in\C^\ell,
\quad
\sigma\in\Sigma_{j,m},
\quad
0\le q<d_{j,m}^{\sigma},
\]
such that every $f\in H_0^2(0,T;\C^\ell)$ satisfying
\begin{equation}
\label{eq:moment-system}
\int_0^T(f''+\gamma_jf)E_{j,q,m}^{\sigma}\,dt
=\zeta_{j,q,m}^{\sigma}
\end{equation}
for every admissible $(j,q,m,\sigma)$ also satisfies
\eqref{eq:terminal}. Moreover, for some $P\ge0$ and $c>0$,
\begin{equation}
\label{eq:mom}
|\zeta_{j,q,m}^{\sigma}|
\le C(1+m)^P e^{-cm^2T}\mathcal E_0.
\end{equation}
\end{proposition}

\begin{proof}
For $z=a,b$, write $z^{(0)}:=z$ and $z^{(1)}:=z'$. For
$1\le s\le\nu_j$, set
$\mathcal K_{j,s}:=
\{k\in\{1,\ldots,g_j\}:r_{j,k}\ge s\},
\quad n_{j,s}:=|\mathcal K_{j,s}|,$
and write $\mathcal K_{j,s}=\{k_1,\ldots,k_{n_{j,s}}\}$. Define
\begin{equation}
\label{eq:active-chains}
\widehat B_{j,s}:=
(\beta_{j,k_1,1}\mid\cdots\mid\beta_{j,k_{n_{j,s}},1})^*
\in\C^{n_{j,s}\times\ell}.
\end{equation}
By \eqref{eq:Bhat-rank}, $\widehat B_{j,s}$ has full row rank. Fix a
right inverse
\[
\widehat C_{j,s}\in\C^{\ell\times n_{j,s}},
\quad
\widehat B_{j,s}\widehat C_{j,s}=I_{n_{j,s}}.
\]
Thus, for every $D\in\C^{n_{j,s}}$,
$\widehat B_{j,s}(\widehat C_{j,s}D)=D.$

This is the only inverse identity used below; in general
$\widehat C_{j,s}\widehat B_{j,s}\ne I_\ell$.
For a prospective control $f$, put $F_j:=f''+\gamma_jf$.
We first isolate the simple-root calculation at the eigenvector level
$s=1$; it fixes the degree-zero targets and provides the model for the
higher Jordan levels, while double temporal roots are treated afterwards
by the confluent analogue.

Fix $(j,m)\notin\mathcal D_\rho$ and set
\[
\Delta_{j,m}:=\lambda_{j,m}^+-\lambda_{j,m}^-,
\quad
\A_{j,m}:=\Delta_{j,m}^{-1}.
\]
Thus  $\Delta_{j,m}\ne0$ by assumption, and $x_m\ne0$.
Recall
\begin{equation}
\label{a}
a_{j,k,s,m}
=-K_{j,m}*\Bigl[
a_{j,k,s-1,m}
+x_m\bigl(\beta_{j,k,s}^*F_j+\beta_{j,k,s-1}^*f\bigr)
\Bigr].
\end{equation}
At $s=1$, $n_{j,1}=g_j$ and
$a_{j,k,0,m}=\beta_{j,k,0}^*=0$. For $d=0,1$, denote 
$g^{(d)}(t)=$ the $d$th derivative, and 
set
\[
\vec\eta_{j,1,m}^{\,d}:=
\bigl(b_{j,k_i,1,m}^{(d)}(T)\bigr)_{1\le i\le n_{j,1}}.
\]
Then \eqref{eq:terminal}, \eqref{a}, and Lemma~\ref{Kconv} give
\[
\vec\eta_{j,1,m}^{\,d}
=\A_{j,m}x_m\widehat B_{j,1}
\int_0^T F_j(t)\Bigl[
(\lambda_{j,m}^+)^d e^{\lambda_{j,m}^+(T-t)}
-(\lambda_{j,m}^-)^d e^{\lambda_{j,m}^-(T-t)}
\Bigr]dt .
\]
Since $\widehat B_{j,1}\widehat C_{j,1}=I_{n_{j,1}}$, it is sufficient to impose for $d=0,1$,
\[
\widehat C_{j,1}\vec\eta_{j,1,m}^{\,d}
=\A_{j,m}x_m
\int_0^T F_j(t)\Bigl[
(\lambda_{j,m}^+)^d e^{\lambda_{j,m}^+(T-t)}
-(\lambda_{j,m}^-)^d e^{\lambda_{j,m}^-(T-t)}
\Bigr]dt,
\]
Indeed, multiplying this lifted identity by $\widehat B_{j,1}$
recovers the terminal system. This is a sufficient lifting, not an equivalent reformulation, since $\ker\widehat B_{j,1}$ may be
nontrivial.
Solving these two lifted equations, define
\[
\zeta_{j,0,m}^{\pm}:=
\frac{
\widehat C_{j,1}\vec\eta_{j,1,m}^{\,1}
-\lambda_{j,m}^{\mp}\widehat C_{j,1}\vec\eta_{j,1,m}^{\,0}
}{x_m}.
\]
Thus
\beq
\int_0^T F_j(t)e^{\lambda_{j,m}^{\pm}(T-t)}dt
=\zeta_{j,0,m}^{\pm}\label{mo1}
\eeq
enforces both terminal conditions at $s=1$. The right inverse
$\widehat C_{j,1}$ lifts the $g_j$ chain equations to two targets in
$\C^\ell$, hence no $k$-index occurs. If $A$ is diagonalizable, this
completes the simple-root construction; otherwise, the same lifting is
iterated after the lower-order terms have been fixed.
We apply integration by parts twice to  \eqref{mo1}, using $f\in H_0^2(0,T)$, to get
$$\int_0^T f(t)e^{\lambda_{j,m}^{\pm}(T-t)}dt
=\frac{\zeta_{j,0,m}^{\pm}}{(\lambda_{j,m}^{\pm})^2+\gamma_j}.
$$

\medskip

We now establish moment equations for $s>1$, which is more complicated. This will be an inductive process for which \eqref{mo1} was the first step.

\noindent
\emph{Moment notation and triangular relation.}
For every admissible $(j,q,m,\sigma)$, set
\begin{equation}
\label{eq:Z-moments}
\begin{aligned}
Z_{j,q,m}^{\sigma}(f)
&:=\int_0^T F_j(t)E_{j,q,m}^{\sigma}(t)\,dt,\\
M_{j,q,m}^{\sigma}(f)
&:=\int_0^T f(t)E_{j,q,m}^{\sigma}(t)\,dt.
\end{aligned}
\end{equation}
Here $Z(f),M(f)$ are the actual moments, whereas $\zeta,\mathcal M$
denote their targets. By $q$-moments, we will refer to moments involving $E^\sigma_{j,q,m}$.
Since
$E_{j,0,m}^{\pm}(t)=e^{\lambda_{j,m}^{\pm}(T-t)}$, the base-level
conditions are precisely
$Z_{j,0,m}^{\pm}(f)=\zeta_{j,0,m}^{\pm}.$
Moreover, $f,f'$ vanish at both endpoints and
\[
\bigl(E_{j,q,m}^{\sigma}\bigr)''
=(\lambda_{j,m}^{\sigma})^2E_{j,q,m}^{\sigma}
+2q\lambda_{j,m}^{\sigma}E_{j,q-1,m}^{\sigma}
+q(q-1)E_{j,q-2,m}^{\sigma}.
\]
Hence two integrations by parts yield
\[
\begin{aligned}
Z_{j,q,m}^{\sigma}(f)
={}&\bigl((\lambda_{j,m}^{\sigma})^2+\gamma_j\bigr)
M_{j,q,m}^{\sigma}(f)\\
&+2q\lambda_{j,m}^{\sigma}M_{j,q-1,m}^{\sigma}(f)
+q(q-1)M_{j,q-2,m}^{\sigma}(f),
\end{aligned}
\]
where moments of negative degree are zero. For prescribed
$Z$-targets, define the corresponding $M$-targets by
\begin{equation}
\label{eq:target-f-moments}
\mathcal M_{j,q,m}^{\sigma}:=
\frac{
\zeta_{j,q,m}^{\sigma}
-2q\lambda_{j,m}^{\sigma}\mathcal M_{j,q-1,m}^{\sigma}
-q(q-1)\mathcal M_{j,q-2,m}^{\sigma}
}{(\lambda_{j,m}^{\sigma})^2+\gamma_j},
\end{equation}
with
$\mathcal M_{j,-1,m}^{\sigma}
=\mathcal M_{j,-2,m}^{\sigma}:=0$.
This is well defined by \eqref{nr2}, and the preceding triangular
identity gives
\[
\begin{gathered}
Z_{j,r,m}^{\sigma}(f)=\zeta_{j,r,m}^{\sigma},
\quad 0\le r\le q,\\
\Longrightarrow\quad
M_{j,r,m}^{\sigma}(f)=\mathcal M_{j,r,m}^{\sigma},
\quad 0\le r\le q.
\end{gathered}
\]
Consequently, at a simple-root level $s\ge2$, the target values of all
$Z$- and $M$-moments of degrees at most $s-2$ are fixed; they determine
the lower-order residual, and only the two $Z$-targets of degree $s-1$
are new. Returning to the
simple-root pair fixed above, fix $2\le s\le\nu_j$, suppose that
$\zeta_{j,q,m}^{\pm}$ are prescribed for $0\le q\le s-2$, and set
$\kappa_{j,s,m}:=
\frac{x_m\A_{j,m}^{s}}{(s-1)!}\ne0.$

For $d=0,1$, \eqref{eq:a-iterated} and Lemma~\ref{Kconv} give
\begin{equation}
\label{eq:leading}
\begin{aligned}
a_{j,k,s,m}^{(d)}(T)
={}&\kappa_{j,s,m}\beta_{j,k,1}^*
\Bigl[(-1)^s(\lambda_{j,m}^+)^d
Z_{j,s-1,m}^+(f)
+(\lambda_{j,m}^-)^d Z_{j,s-1,m}^-(f)\Bigr]\\
&+\mathcal L_{j,k,s,m}^d(f).
\end{aligned}
\end{equation}
Here
\[
\mathcal L_{j,k,s,m}^d(f)
=\sum_{\sigma\in\{+,-\}}\sum_{q=0}^{s-2}
\Bigl(
A_{j,k,s,m,q}^{d,\sigma}Z_{j,q,m}^{\sigma}(f)
+B_{j,k,s,m,q}^{d,\sigma}M_{j,q,m}^{\sigma}(f)
\Bigr),
\]
where
$A_{j,k,s,m,q}^{d,\sigma},B_{j,k,s,m,q}^{d,\sigma}
\in\C^{1\times\ell}$ are determined by
\eqref{eq:a-iterated} and Lemma~\ref{Kconv}, and have at most
polynomial growth in $m$.
Replacing $Z_{j,q,m}^{\sigma}(f)$ and
$M_{j,q,m}^{\sigma}(f)$ in the preceding expression by
$\zeta_{j,q,m}^{\sigma}$ and $\mathcal M_{j,q,m}^{\sigma}$,
respectively, defines the known scalar
$\mathcal L_{j,k,s,m}^{d,\mathrm{tar}}$ by the induction hypothesis.

Set
\[
\begin{aligned}
D_{j,s,m}^d
&:=\Bigl(
-b_{j,k_i,s,m}^{(d)}(T)
-\mathcal L_{j,k_i,s,m}^{d,\mathrm{tar}}
\Bigr)_{i=1}^{n_{j,s}}
\in\C^{n_{j,s}},\\
U_{j,s,m}^d
&:=\kappa_{j,s,m}^{-1}\widehat C_{j,s}D_{j,s,m}^d
\in\C^\ell.
\end{aligned}
\]
It is sufficient to impose
\begin{align}
(-1)^s\zeta_{j,s-1,m}^++\zeta_{j,s-1,m}^-
&=U_{j,s,m}^0,
\label{eq:lifted}\\
(-1)^s\lambda_{j,m}^+\zeta_{j,s-1,m}^+
+\lambda_{j,m}^-\zeta_{j,s-1,m}^-
&=U_{j,s,m}^1.
\label{eq:lifted1}
\end{align}
Indeed,
$\kappa_{j,s,m}\widehat B_{j,s}U_{j,s,m}^d=D_{j,s,m}^d$.
Moreover, whenever
\[
Z_{j,q,m}^{\sigma}(f)=\zeta_{j,q,m}^{\sigma},
\quad
\sigma\in\{+,-\},\quad 0\le q\le s-1,
\]
the triangular relation gives
$M_{j,q,m}^{\sigma}(f)=\mathcal M_{j,q,m}^{\sigma}$ for
$0\le q\le s-2$, and hence
$\mathcal L_{j,k,s,m}^d(f)
=\mathcal L_{j,k,s,m}^{d,\mathrm{tar}}$.
Thus \eqref{eq:leading} and \eqref{eq:lifted}--\eqref{eq:lifted1}
yield
\[
a_{j,k,s,m}^{(d)}(T)=-b_{j,k,s,m}^{(d)}(T),
\quad d=0,1.
\]

Since $\Delta_{j,m}\ne0$, the lifted system has the solution
\begin{align}
\zeta_{j,s-1,m}^+
&=(-1)^s\frac{U_{j,s,m}^1
-\lambda_{j,m}^-U_{j,s,m}^0}{\Delta_{j,m}},
\label{eq:level-s-Z-plus}\\
\zeta_{j,s-1,m}^-
&=\frac{\lambda_{j,m}^+U_{j,s,m}^0
-U_{j,s,m}^1}{\Delta_{j,m}}.
\label{eq:level-s-Z-minus}
\end{align}

\medskip
\noindent
\emph{Double roots.}
Assume $(j,m)\in\mathcal D_\rho$, write
$\lambda:=\lambda_{j,m}^c$, and fix $1\le s\le\nu_j$.
Suppose that the targets $\zeta_{j,q,m}^c$ have already been
prescribed for $0\le q\le2s-3$, with an empty range when $s=1$.
Set
$\kappa_{j,s,m}^c
:=
\frac{(-1)^s x_m}{(2s-1)!}\ne0.$
By \eqref{eq:double-convolution},
\begin{align}
a_{j,k,s,m}(T)
&=
\kappa_{j,s,m}^c\beta_{j,k,1}^*
Z_{j,2s-1,m}^c(f)
+\mathcal L_{j,k,s,m}^{c,0}(f),
\label{eq:double-leading-position}\\
a'_{j,k,s,m}(T)
&=
\kappa_{j,s,m}^c\beta_{j,k,1}^*
\Bigl(
\lambda Z_{j,2s-1,m}^c(f)
\notag\\[-0.2em]
&\qquad
+(2s-1)Z_{j,2s-2,m}^c(f)
\Bigr)
\notag\\
&\qquad
+\mathcal L_{j,k,s,m}^{c,1}(f).
\label{eq:double-leading-velocity}
\end{align}
Here $\mathcal L_{j,k,s,m}^{c,d}(f)$ depends only on the
moments $Z_{j,q,m}^c(f)$ and $M_{j,q,m}^c(f)$ with
$q\le2s-3$, and vanishes when $s=1$. Let
$\mathcal L_{j,k,s,m}^{c,d,\mathrm{tar}}$ be obtained by replacing
these moments by their targets
$\zeta_{j,q,m}^c$ and $\mathcal M_{j,q,m}^c$.

Define
\[
\begin{aligned}
D_{j,s,m}^{c,d}
&:=
\left(
-b_{j,k_i,s,m}^{(d)}(T)
-\mathcal L_{j,k_i,s,m}^{c,d,\mathrm{tar}}
\right)_{i=1}^{n_{j,s}},\\
U_{j,s,m}^{c,d}
&:=
(\kappa_{j,s,m}^c)^{-1}
\widehat C_{j,s}D_{j,s,m}^{c,d}.
\end{aligned}
\]
It is sufficient to impose
\begin{align}
\zeta_{j,2s-1,m}^c
&=U_{j,s,m}^{c,0},
\label{eq:double-lifted-position}\\
\lambda\zeta_{j,2s-1,m}^c
+(2s-1)\zeta_{j,2s-2,m}^c
&=U_{j,s,m}^{c,1}.
\label{eq:double-lifted-velocity}
\end{align}
Indeed, whenever the lower-degree moment conditions hold,
\[
\mathcal L_{j,k,s,m}^{c,d}(f)
=
\mathcal L_{j,k,s,m}^{c,d,\mathrm{tar}},
\qquad
\kappa_{j,s,m}^c
\widehat B_{j,s}U_{j,s,m}^{c,d}
=
D_{j,s,m}^{c,d}.
\]
Thus the terminal equations hold at level $s$, and
\begin{align}
\zeta_{j,2s-1,m}^c
&=U_{j,s,m}^{c,0},
\label{eq:double-Z-high}\\
\zeta_{j,2s-2,m}^c
&=
\frac{
U_{j,s,m}^{c,1}
-\lambda U_{j,s,m}^{c,0}
}{2s-1}.
\label{eq:double-Z-low}
\end{align}
This completes the recursive prescription in both cases. For any
$f$ satisfying \eqref{eq:moment-system}, two integrations by parts and
induction on the degree give
$\int_0^T fE_{j,q,m}^{\sigma}=\mathcal M_{j,q,m}^{\sigma}$. Hence the
residual terms are exactly those used in the construction above, and the
lifted equations imply \eqref{eq:terminal}.

\medskip
\noindent
\emph{Estimates.}
By~(3.19), after decreasing $c>0$ if necessary, there exists
$\alpha_0\geq 0$ such that, uniformly in the admissible indices,
\[
 \bigl|b^{(d)}_{j,k,s,m}(T)\bigr|
 \leq C(1+m)^{\alpha_0}e^{-cm^2T}E_0,
 \quad d=0,1.
\]
We prove by induction on $s$ that all the targets already
constructed, together with the associated $M$-targets, satisfy the same estimate with some exponent $\alpha_s$.
For $s=1$, this follows from~(4.6) and~(4.8), using, for large
$m$,
\[
 |x_m|^{-1}=O(m),\quad
|\lambda^\pm_{j,m}|=O(m^2),\quad
 \bigl|(\lambda^\pm_{j,m})^2+\gamma_j\bigr|^{-1}=O(m^{-4}).
\]
Assume the assertion up to level $s-1$. By~(3.13) and Lemma~6,
the coefficients entering $L^{d,\mathrm{tar}}_{j,k,s,m}$ are
finite combinations of powers of $x_m$, $C_{j,m}$ and
$\lambda^\pm_{j,m}$; hence they have at most polynomial growth
in $m$. Moreover,
\[
 |\kappa_{j,s,m}^{-1}|
 \leq C(1+m)^{2s+1},
 \quad
 |\Delta_{j,m}^{-1}|
 \leq C(1+m)^{-2}
\]
for all sufficiently large $m$. Therefore
$L^{d,\mathrm{tar}}_{j,k,s,m}$, $D^d_{j,s,m}$ and
$U^d_{j,s,m}$ satisfy the same Gaussian estimate, with a
possibly larger polynomial exponent. Formulas~(4.12)--(4.13)
and then~(4.8) yield the assertion at level $s$.
The remaining simple modes and all double-root modes form a
finite set. Since the Jordan-chain lengths are bounded, taking
 $P=\max_s\alpha_s$
and enlarging $C$ proves~(4.3).
\end{proof}

\begin{proof}[Proof of Theorem~\ref{thm:subcritical}]
\noindent
\emph{ Solution of the moment problem.}
Let
\[
\mathcal I:=\left\{(j,q,m,\sigma):
1\le j\le p_0,\ m\ge1,\ \sigma\in\Sigma_{j,m},
\ 0\le q<d_{j,m}^{\sigma}\right\}.
\]
By Proposition~\ref{prop:moment}, it remains to construct a single
$f\in H_0^2(0,T;\C^\ell)$ satisfying \eqref{eq:moment-system} for all
indices in $\mathcal I$.

Recall that $\alpha$ was chosen so that
\[
(\lambda_{j,m}^{\sigma})^2+\alpha\ne0
\quad
((j,q,m,\sigma)\in\mathcal I).
\]
Moreover, \eqref{nr2} gives
$(\lambda_{j,m}^{\sigma})^2+\gamma_j\ne0
\quad
((j,q,m,\sigma)\in\mathcal I),$
so that the following recursion is well defined.
For each index, write $\lambda=\lambda_{j,m}^{\sigma}$, set
$Y_{j,-1,m}^{\sigma}=Y_{j,-2,m}^{\sigma}:=0$, and define
\begin{equation}
\label{eq:Y-recursion}
Y_{j,q,m}^{\sigma}:=
\frac{\zeta_{j,q,m}^{\sigma}
-2q\lambda Y_{j,q-1,m}^{\sigma}
-q(q-1)Y_{j,q-2,m}^{\sigma}}
{\lambda^2+\gamma_j}.
\end{equation}
At this stage, the $Y_{j,q,m}^{\sigma}$ are only auxiliary target
values for the moments of the control; no control has yet been assumed
to realize them. Define
\begin{equation}
\label{eq:Xi}
\Xi_{j,q,m}^{\sigma}:=
(\lambda^2+\alpha)Y_{j,q,m}^{\sigma}
+2q\lambda Y_{j,q-1,m}^{\sigma}
+q(q-1)Y_{j,q-2,m}^{\sigma}.
\end{equation}
By \eqref{eq:Y-recursion}, this is equivalently
\begin{equation}
\label{eq:Xi-correction}
\Xi_{j,q,m}^{\sigma}
=
\zeta_{j,q,m}^{\sigma}
-(\gamma_j-\alpha)Y_{j,q,m}^{\sigma}.
\end{equation}
This identity corresponds to
$f''+\gamma_jf
=
(f''+\alpha f)+(\gamma_j-\alpha)f.$
Using the block-biorthogonal family from
\eqref{eq:block2}--\eqref{eq:block4}, including the confluent functions
when $(j,m)\in D_\rho$, set
\begin{equation}
\label{eq:G-series}
G(t):=
\sum_{(j,q,m,\sigma)\in\mathcal I}
\Xi_{j,q,m}^{\sigma}g_{j,q,m}^{\sigma}(t).
\end{equation}
For all sufficiently large modes,
\[
|\lambda_{j,m}^{\sigma}|\asymp m^2,
\quad
|(\lambda_{j,m}^{\sigma})^2+\gamma_j|\asymp m^4,
\]
while the finitely many remaining denominators are nonzero by
\eqref{nr2}. Hence \eqref{eq:mom}, \eqref{eq:Y-recursion}, and
\eqref{eq:Xi} give
\begin{equation}
\label{eq:Xi-decay}
|\Xi_{j,q,m}^{\sigma}|
\le C_T(1+m)^P e^{-cm^2T}\mathcal E_0.
\end{equation}
Using \eqref{eq:block4},
\[
\sum_{(j,q,m,\sigma)\in\mathcal I}
|\Xi_{j,q,m}^{\sigma}|\ 
\|g_{j,q,m}^{\sigma}\|_{L^2}
\le
C_T\mathcal E_0
\sum_{m\ge1}(1+m)^{P+M}e^{-cm^2T+C_Tm}
<\infty.
\]
Thus \eqref{eq:G-series} converges absolutely in $L^2$ and
$\|G\|_{L^2}\le C_T\mathcal E_0$. By biorthogonality,
\begin{equation}
\label{eq:G-moments}
\int_0^TGE_{j,q,m}^{\sigma}\,dt
=\Xi_{j,q,m}^{\sigma},
\quad
\int_0^TG\chi_\alpha^d\,dt=0,
\quad d=0,1.
\end{equation}
The first identity is an equality in $\C^\ell$ and therefore holds
componentwise: all scalar equations contained in every prescribed
vector moment are satisfied simultaneously.

Finally,  we now recover the physical control from the single reference equation
\begin{equation}
\label{eq:reference-ode}
f''+\alpha f=G,
\quad
f(0)=f'(0)=0.
\end{equation}
Writing $\omega^2=-\alpha$, its solution is
\[
f(t)=
\begin{cases}
\displaystyle
\frac1{2\omega}\int_0^t
\bigl(e^{\omega(t-r)}-e^{-\omega(t-r)}\bigr)G(r)\,dr,
&\omega\ne0,\\[3mm]
\displaystyle
\int_0^t(t-r)G(r)\,dr,
&\omega=0.
\end{cases}
\]
If $\omega\ne0$, put
$M_\pm:=\int_0^T e^{\pm\omega(T-r)}G(r)\,dr.$
Then
\[
f(T)=\frac{M_+-M_-}{2\omega},
\quad
f'(T)=\frac{M_++M_-}{2}.
\]
The auxiliary conditions in \eqref{eq:G-moments} give
$M_+=M_-=0$. If $\omega=0$, they instead give
\[
f(T)=\int_0^T(T-r)G(r)\,dr=0,
\quad
f'(T)=\int_0^TG(r)\,dr=0.
\]
Consequently, $f(T)=f'(T)=0$ and
$f\in H_0^2(0,T;\C^\ell)$.
It remains to verify the original moment system. Put
\[
\widehat Y_{j,q,m}^{\sigma}:=
\int_0^Tf(t)E_{j,q,m}^{\sigma}(t)\,dt,
\quad
\widehat Y_{j,-1,m}^{\sigma}
=\widehat Y_{j,-2,m}^{\sigma}:=0.
\]
Two integrations by parts, using all four endpoint traces, yield
\[
\Xi_{j,q,m}^{\sigma}
=(\lambda^2+\alpha)\widehat Y_{j,q,m}^{\sigma}
+2q\lambda\widehat Y_{j,q-1,m}^{\sigma}
+q(q-1)\widehat Y_{j,q-2,m}^{\sigma}.
\]
Comparison with \eqref{eq:Xi} and induction on $q$ give
$\widehat Y_{j,q,m}^{\sigma}=Y_{j,q,m}^{\sigma}
\hbox{ }((j,q,m,\sigma)\in\mathcal I),$
because $\lambda^2+\alpha\ne0$. Consequently,
\begin{align*} 
&\int_0^T(f''+\gamma_jf)E_{j,q,m}^{\sigma}\,dt = \Xi_{j,q,m}^{\sigma} +(\gamma_j-\alpha)Y_{j,q,m}^{\sigma} \\ 
&= (\lambda^2+\gamma_j)Y_{j,q,m}^{\sigma} +2q\lambda Y_{j,q-1,m}^{\sigma} +q(q-1)Y_{j,q-2,m}^{\sigma} = \zeta_{j,q,m}^{\sigma}, 
\end{align*}
where the last equality is \eqref{eq:Y-recursion}. Hence the single control $f$ satisfies every vector moment equation constructed above, and therefore every scalar component of those equations.
Proposition~\ref{prop:moment} gives \eqref{eq:terminal}.
Since $f(T)=f'(T)=0$, the lifting term and its time derivative vanish at
$T$. Hence
$u(\cdot,T)=0,
\quad
u_t(\cdot,T)=0.$
Finally, the solution operator in \eqref{eq:reference-ode} maps $L^2$
continuously into $H^2$, and the preceding estimate on $G$ gives
\[
\|f\|_{H^2(0,T;\C^\ell)}
\le C_T\|G\|_{L^2(0,T;\C^\ell)}
\le C_T\mathcal E_0.
\]
This is \eqref{eq:control}.
\end{proof}
\section{The critical value \texorpdfstring{$\rho=2$}{rho=2}}
\label{sec:critical}
At the critical value \(\rho=2\), the principal differential expression
factorizes as \((\partial_t+L)^2\). Since a controlled trajectory need not
belong to \(D(L)\), we define the second component using the differential
expression \(-u_{xx}\). Set \(Q_T:=(0,\pi)\times(0,T)\) and
\[
w:=u_t-u_{xx},\quad
Y:=\binom{u}{w},\quad
U_f:=\binom{Bf}{Bf'},\quad
\mathcal M:=
\begin{pmatrix}
0&I_N\\
-A&0
\end{pmatrix}.
\]
For smooth compatible data, \(u_t=u_{xx}+w\) and
\(w_t=w_{xx}-Au\). Hence
\begin{equation}\label{eq:critical-heat-system}
\left\{
\begin{aligned}
Y_t&=Y_{xx}+\mathcal M Y
&&\text{in }Q_T,\\
Y(0,t)&=0,\quad Y(\pi,t)=U_f(t)
&&t\in(0,T),\\
Y(\cdot,0)&=
(u^0,u^1-u^0_{xx})^\top
=
(u^0,u^1+Lu^0)^\top.
\end{aligned}
\right.
\end{equation}
Indeed, the original boundary conditions give
\(w(0,t)=0\) and \(w(\pi,t)=Bf'(t)\).
To justify this reduction in the stated solution class, let
\(h(x):=x/\pi\) and set \(Z:=Y-hU_f\). Since
\(f\in H_0^2(0,T;\mathbb C^\ell)\), we have
\(f(0)=f'(0)=0\), and \(Z\) solves
\begin{equation}\label{eq:critical-lifting}
\left\{
\begin{aligned}
Z_t&=Z_{xx}+\mathcal MZ
-h(0,Bf''+ABf)^\top
&&\text{in }Q_T,\\
Z(0,t)&=Z(\pi,t)=0
&&t\in(0,T),\\
Z(\cdot,0)&=
(u^0,u^1+Lu^0)^\top.
\end{aligned}
\right.
\end{equation}
The source in \eqref{eq:critical-lifting} belongs to
$L^2(0,T;\mathbf X^0\times\mathbf X^0)$. Hence parabolic maximal
regularity gives
\[
Z\in
\left[
H^1(0,T;\mathbf X^0)
\cap L^2(0,T;\mathbf X^2)
\cap C([0,T];\mathbf X^1)
\right]^2.
\]
Writing $Z=(v,\zeta)^\top$, the first component yields
$\zeta=v_t+Lv$. Since $\zeta\in L^2(0,T;\mathbf X^2)$ and
$v(0)=u^0\in\mathbf X^3$, a second maximal-regularity argument gives
\[
v\in
H^1(0,T;\mathbf X^2)
\cap L^2(0,T;\mathbf X^4)
\cap C([0,T];\mathbf X^3)\hbox{, }
v_t\in C([0,T];\mathbf X^1).
\]

Thus, setting $u=v+hBf$, we recover
$Y=Z+hU_f=(u,u_t-u_{xx})^\top.$
Moreover, since $f(0)=f'(0)=0$,
\[
u(0)=u^0,
\quad
u_t(0)=\zeta(0)-Lu^0=u^1,
\]
and $u$ satisfies the prescribed boundary conditions. Here only the
differential identities $h''=h^{(4)}=0$ are used, since
$h\notin D(L)$; the argument extends from smooth compatible data by
density.

Finally, $f(T)=f'(T)=0$ implies $Y(T)=Z(T)$, and
\begin{equation}\label{equivalence-proof}
Y(T)=0
\quad\Longleftrightarrow\quad
u(T)=u_t(T)=0.
\end{equation}
Hence the beam equation is exactly reduced to a $2N$-component heat
system with linked boundary inputs $(Bf,Bf')$, a constraint not covered
by the independent-control result of~\cite{AKBGT11}.
\begin{lemma}
\label{lem:biorthogo}
Let $\mathcal V\subset i\mathbb R$ be finite, without repeated values, and set
 $\Lambda_{m,\nu}=m^2-\nu,
 \quad m\ge1,\quad \nu\in\mathcal V .$
For every $T>0$, there exist functions
$\Theta_{m,\nu,q}\in L^2(0,T)$, $q=0,1$, satisfying
\begin{align}
 \int_0^T t^r e^{-\Lambda_{n,\widetilde\nu}t}
 \Theta_{m,\nu,q}(t)\,dt
 &=
 \delta_{mn}\delta_{\nu\widetilde\nu}\delta_{qr},
 \label{eq:critical-biorthogonality}\\
 \int_0^T\Theta_{m,\nu,q}(t)\,dt
 &=
 \int_0^T(T-t)\Theta_{m,\nu,q}(t)\,dt=0,
 \label{eq:critical-finite-annihilation}
\end{align}
for $r=0,1$. Moreover, for every $\varepsilon>0$,
$\|\Theta_{m,\nu,q}\|_{L^2(0,T)}
 \le C_{T,\varepsilon}e^{\varepsilon m^2}.$
Consequently, if $d_{m,\nu,q}\in\mathbb C^\ell$ satisfy
 $\sum_{m,\nu,q}e^{\varepsilon m^2}
 \lvert d_{m,\nu,q}\rvert<\infty$
for some $\varepsilon>0$, then there exists
$G\in L^2(0,T;\mathbb C^\ell)$ such that
\begin{align*}
 \int_0^Tt^qe^{-\Lambda_{m,\nu}t}G(t)\,dt
 =d_{m,\nu,q},
 \int_0^TG(t)\,dt
 =\int_0^T(T-t)G(t)\,dt=0.
\end{align*}
\end{lemma}
The proof of this lemma, an application of \cite[Theorem~1.2]{AKBGT11} that resembles the proof of Lemma 2, and is left to the reader.
\begin{proof}[Proof of Theorem~\ref{thm:critical}]
With $\varphi_m$ as above, set
 $\mu_m:=m^2$,
$ \kappa_m:=\varphi_m'(\pi)
 =(-1)^m\sqrt{\frac{2}{\pi}}\,m.$
 
\medskip
\noindent\textbf{Step 1: reduction to the moment equations.}
The moment conditions are obtained by testing the state equation against
solutions of the adjoint system. Let \(P\) solve
\begin{equation}\label{eq:critical-adjoint}
\left\{
\begin{aligned}
-P_t&=P_{xx}+\mathcal M^*P
&&\text{in }(0,\pi)\times(0,T),\\
P(0,t)&=P(\pi,t)=0
&&t\in(0,T).
\end{aligned}
\right.
\end{equation}
Green's formula gives for \(Y(\cdot,T)=0\),
\begin{equation}\label{eq:critical4}
\int_0^T
\left\langle
\binom{Bf(t)}{Bf'(t)},P_x(\pi,t)
\right\rangle\,dt
=
\langle Y^0,P(\cdot,0)\rangle .
\end{equation}
Conversely, if \eqref{eq:critical4} holds for every adjoint solution
whose terminal datum belongs to the algebraic span of the adjoint modes,
Green's formula gives
$\langle Y(\cdot,T),P(\cdot,T)\rangle=0$
on a dense subspace of $L^2(0,\pi;\mathbb C^{2N})$. Hence $Y(\cdot,T)=0$.
Thus null controllability is equivalent to the family of moment identities obtained from \eqref{eq:critical4}.
For \(r=0,1\), set
\[
u_m^r:=\int_0^\pi u^r(x)\varphi_m(x)\,dx,\quad
Y_m^0:=(u_m^0,u_m^1+\mu_m u_m^0)^\top.
\]
Thus \(Y_m^0\) is the \(m\)-th Fourier coefficient of \(Y^0\).

For \(\gamma\geq0\), set
$E_\gamma:=\ker(A-\gamma I_N),\quad
g_\gamma:=\dim E_\gamma.$
For every \(\gamma\in\sigma(A)\), choose an orthonormal basis
\(\{b_{\gamma,k}:1\leq k\leq g_\gamma\}\) of \(E_\gamma\), and define
\[
\mathcal O_\gamma:\mathbb C^\ell\longrightarrow\mathbb C^{g_\gamma},
\quad
\mathcal O_\gamma z
:=
\bigl(
\langle z,B^*b_{\gamma,k}\rangle_{\mathbb C^\ell}
\bigr)_{k=1}^{g_\gamma}.
\]
Since \(A=A^*\), condition \eqref{eq:Hautus} implies that
\(B^*|_{E_\gamma}\) is injective. Hence \(\mathcal O_\gamma\) is
surjective. Fix a right inverse
\[
\mathcal R_\gamma:\mathbb C^{g_\gamma}\longrightarrow\mathbb C^\ell,
\quad
\mathcal O_\gamma\mathcal R_\gamma=I_{g_\gamma}.
\]
\medskip
\noindent\textbf{Step 2: modes associated with \(\gamma>0\).}
Let
\[
\gamma\in\sigma(A)\cap(0,\infty),
\quad
\nu\in\{-i\sqrt\gamma,i\sqrt\gamma\},
\quad
b\in E_\gamma,
\]
and set
$z_{\nu,b}:=(-\nu b,b)^\top.$
Since
$\mathcal M^*z_{\nu,b}=-\nu z_{\nu,b},$
the function
$P_{m,\nu,b}(x,t)
:=
\varphi_m(x)e^{(-\mu_m-\nu)(T-t)}z_{\nu,b}$
is a backward adjoint mode.
Set $a_{m,\nu}:=-\mu_m+\nu,
\quad
F(s):=f(T-s).$
Testing \eqref{eq:critical4} with \(P_{m,\nu,b}\) and arguing as in
Section~4, using \(f(0)=f(T)=0\), gives
\[
\kappa_m(-\mu_m+2\nu)
\int_0^T
\langle F(s),B^*b\rangle_{\mathbb C^\ell}
e^{a_{m,\nu}s}\,ds
=
e^{a_{m,\nu}T}
\langle Y_m^0,z_{\nu,b}\rangle_{\mathbb C^{2N}}.
\]
Applying this identity to the basis
\(\{b_{\gamma,k}\}_{k=1}^{g_\gamma}\), we obtain
\[
\mathcal O_\gamma
\left(
\int_0^T F(s)e^{a_{m,\nu}s}\,ds
\right)
=
\frac{e^{a_{m,\nu}T}}
{\kappa_m(-\mu_m+2\nu)}
\bigl(
\langle Y_m^0,z_{\nu,b_{\gamma,k}}\rangle
\bigr)_{k=1}^{g_\gamma}.
\]
Since
\(\mathcal O_\gamma\mathcal R_\gamma=I_{g_\gamma}\),
it is sufficient to impose
$\int_0^T F(s)e^{a_{m,\nu}s}\,ds
=
I_{m,\nu,0},$
where
\begin{equation}\label{eq:positiv1}
I_{m,\nu,0}
:=
\mathcal R_\gamma
\left[
\frac{e^{a_{m,\nu}T}}
{\kappa_m(-\mu_m+2\nu)}
\bigl(
\langle Y_m^0,z_{\nu,b_{\gamma,k}}\rangle
\bigr)_{k=1}^{g_\gamma}
\right].
\end{equation}
As required for the application of Lemma~9, we also set
$I_{m,\nu,1}:=0.$

\medskip
\noindent\textbf{Step 3: the zero eigenvalue.}
Assume \(0\in\sigma(A)\). For \(b\in E_0\), set
$z_{0,b}:=(0,b)^\top,
z_{1,b}:=(b,0)^\top.$
Then
$\mathcal M^*z_{0,b}=0$, $\mathcal M^*z_{1,b}=z_{0,b},$
so the associated backward adjoint modes are
$\varphi_m(x)e^{-\mu_m(T-t)}z_{0,b}$
and
$\varphi_m(x)e^{-\mu_m(T-t)}
\bigl(z_{1,b}+(T-t)z_{0,b}\bigr).$
Using \(F(s)=f(T-s)\), introduce
$I_{m,0,q}
:=
\int_0^T s^qF(s)e^{-\mu_ms}\,ds\hbox{, }
 q=0,1.$
As in Step~2, integration by parts and
\(F(0)=F(T)=0\) give
\begin{align*}
-\int_0^T F'(s)e^{-\mu_ms}\,ds
&=-\mu_m I_{m,0,0},\\
\int_0^T\bigl(F(s)-sF'(s)\bigr)e^{-\mu_ms}\,ds
&=2I_{m,0,0}-\mu_m I_{m,0,1}.
\end{align*}
The two terms correspond to \(z_{1,b}\) and \(z_{0,b}\), respectively. Set
\[\mathcal{A}_m := 
\begin{pmatrix}
2I_{g_0} & -\mu_m I_{g_0}\\
-\mu_m I_{g_0} & 0_{g_0}
\end{pmatrix},
\quad
\mathcal{V}_m :=
\begin{pmatrix}
 \bigl(\langle Y_m^0, z_{1,b_{0,k}}+Tz_{0,b_{0,k}}\rangle\bigr)_{k=1}^{g_0} \\
 \bigl(\langle Y_m^0,z_{0,b_{0,k}}\rangle\bigr)_{k=1}^{g_0}
\end{pmatrix},
\]
and $\mathcal{O}_{0,m}^T :=\begin{pmatrix}
\mathcal{O}_0 I_{m,0,0} & \mathcal{O}_0 I_{m,0,1}
\end{pmatrix}.$
Then \eqref{eq:critical4} reduces to
$\kappa_m\mathcal A_m\mathcal O_{0,m}
   =e^{-\mu_mT}\mathcal V_m.$
Since
$\det\mathcal A_m=(-\mu_m^2)^{g_0}\neq0,$
we may set
\[
\binom{\mathbf i_{m,0}}{\mathbf i_{m,1}}
 :=\frac{e^{-\mu_mT}}{\kappa_m}
   \mathcal A_m^{-1}\mathcal V_m,
\quad
I_{m,0,q}:=\mathcal R_0\mathbf i_{m,q},
\quad q=0,1.
\]
If \(0\notin\sigma(A)\), we simply prescribe
$I_{m,0,0}=I_{m,0,1}=0.$

\medskip
\noindent\textbf{Step 4: solution of the moment problem and control construction.}
Set
$\mathcal V
:=
\{\pm i\sqrt\gamma:\gamma\in\sigma(A),\ \gamma>0\},$
adding \(0\) when \(0\in\sigma(A)\), and let
$\Lambda_{m,\nu}:=\mu_m-\nu$,
$a_{m,\nu}:=-\Lambda_{m,\nu}.$
With \(I_{m,\nu,-1}:=0\), define
\begin{equation}\label{eq:critical}
d_{m,\nu,q}
:=
a_{m,\nu}^2I_{m,\nu,q}
+2q\,a_{m,\nu}I_{m,\nu,q-1},
\quad q=0,1.
\end{equation}
The estimates established above give
\begin{equation}\label{eq:critical-d}
|d_{m,\nu,q}|
\le C_Te^{-\mu_mT}\mathcal E_0,
\quad
\mathcal E_0
:=
\|u^0\|_{\mathbf X^3}
+\|u^1\|_{\mathbf X^1}.
\end{equation}
Hence, for every \(0<\varepsilon<T\),
$\sum_{m,\nu,q}e^{\varepsilon m^2}|d_{m,\nu,q}|<\infty.$
Lemma~\ref{lem:biorthogo} therefore provides
\(G\in L^2(0,T;\mathbb C^\ell)\) such that
\begin{equation}\label{eq:critical1}
\int_0^Ts^qe^{a_{m,\nu}s}G(s)\,ds
=d_{m,\nu,q},
\quad q=0,1,
\end{equation}
together with
$\int_0^TG(s)\,ds
=
\int_0^T(T-s)G(s)\,ds
=0$, $\|G\|_{L^2}\le C_T\mathcal E_0.$
Define
\begin{equation}\label{eq:critical3}
F(t):=\int_0^t(t-s)G(s)\,ds,
\quad
f(t):=F(T-t).
\end{equation}
Then \(F''=G\). Its definition gives
\(F(0)=F'(0)=0\), while the two auxiliary conditions give
\(F(T)=F'(T)=0\). Thus
$f\in H_0^2(0,T;\mathbb C^\ell).$
Two integrations by parts, followed by comparison with
\eqref{eq:critical}, show successively for \(q=0,1\) that
$\int_0^Ts^qF(s)e^{a_{m,\nu}s}\,ds
=I_{m,\nu,q},$
since \(a_{m,\nu}\neq0\). Hence all the required moment identities hold.
By completeness of the adjoint modes,
\[
u(\cdot,T)=u_t(\cdot,T)=0,
\quad
\|f\|_{H^2}
\le C_T\|G\|_{L^2}
\le C_T\mathcal E_0.
\]
Therefore the system is null controllable for every \(T>0\).
\end{proof}

\section{Scope and limitations}
Condition~{\rm (NR$_1$)} excludes inter-pair spectral
collisions, whose treatment would require a block Hautus
condition for the full generator. Intra-pair double roots and
the small spectral gaps considered here are handled by the
confluent construction and
Lemma~\ref{lem:biorthogonality}. For \(\rho>2\), factorization leads to a parabolic system with
distinct diffusivities and linked boundary inputs \((Bf,Bf')\). In the uncoupled scalar Dirichlet case \(N=\ell=1\), \(A=0\), and \(B=1\), this overdamped regime was analyzed in \cite{AE}: exact collisions between the two temporal branches may preclude even approximate controllability, whereas spectral condensation
may give rise to a strictly positive minimal null-control time. Extending this analysis to the coupled setting would require accounting simultaneously for interbranch collisions and condensation, the spectrum and Jordan structure of \(A\), and the control geometry induced by \(B\), and lies beyond the scope of the present work. At \(\rho=0\), the Gaussian high-frequency decay underlying the moment construction is lost, so the present
argument no longer applies.

\textbf{Funding:} The second author conducted this research during an Alexander von HumboldtFellowship at the Chair for Dynamics, Control, Machine Learning and Numerics, Friedrich-Alexander-Universit\"at Erlangen--N\"urnberg.

\textbf{Conflicts of Interest:}
The authors declare no conflicts of interest.

\end{document}